\documentclass[11pt,a4paper, twoside, leqno, final]{arxiv2}
\usepackage[shortalphabetic]{amsrefs}
\usepackage[english]{babel}
\usepackage[utf8]{inputenc}
\usepackage{fancyhdr}
\usepackage{hyperref}
\usepackage{indentfirst}
\usepackage{graphicx}
\usepackage{newlfont}
\usepackage[active]{srcltx}
\usepackage{mathrsfs}
\usepackage{amssymb}
\usepackage{amsmath}
\usepackage{latexsym}
\usepackage{amsthm}
\usepackage[all]{xy}
\usepackage{booktabs}
\usepackage{multicol}
\usepackage[dvipsnames,usenames]{color}
\usepackage[refpage,prefix]{nomencl}
\usepackage{comandi}

\numberwithin{equation}{section}

\theoremstyle{plain}                    
 
\newtheorem{lem}[thm]{Lemma}            

\newtheorem*{thm*}{Theorem A}

\newtheorem*{teo*}{Theorem}

\theoremstyle{definition} 

\theoremstyle{remark}                   
 
\title{On the tetracanonical map of varieties of general type and
maximal
Albanese dimension}
\titlemark{On the tetracanonical map of varieties of maximal
Albanese dimension}


\author{Sofia Tirabassi}
\address{
Dipartimento di Matematica\\
Università degli Studi  Roma Tre\\
\email{tirabass@mat.uniroma3.it}
}

\date{}

\definecolor{grigio}{gray}{.8}
\definecolor{viola}{rgb}{.4,0,.6}
\definecolor{rosso}{rgb}{.7,0,.1}
\definecolor{verde}{rgb}{0,.3,0}

\makenomenclature
\begin{document}
 \maketitle
\begin{abstract}
 \noindent We prove the birationality of the $4$-canonical map
of smooth projective varieties of
general type and maximal Albanese dimension.
\end{abstract}

\section{Introduction}
\nocite{PPI}
\nocite{PPI}
An interesting problem in birational geometry is,
given $Z$ a projective variety of general type, to understand for
which $n$ the $n$-th
pluricanonical
linear system $|\omega_Z^{\otimes n}|$ yields a birational map.\\
\indent For varieties of 
\emph{maximal
Albanese
dimension} (i.e. the Albanese map $Z\longrightarrow
\mathrm{Alb}(Z)$ is generically finite), this study brought some
surprising result, since  Chen
and Hacon \cite{CH2007} proved that, independently on the
dimension,
$\omega_Z^{\otimes n}$ is birational for every $n\ge 6$. More
recently, it
was shown by
Jiang \cite{Jiang2009}, applying ideas from \cite{PP3}, that, if
$Z$ is as above, the
pentacanonical map
$\varphi_{|\omega_Z^{\otimes 5}|}$  is birational.
The main result of this paper, whose proof also uses
$M$-regularity
techniques introduced by Pareschi and Popa in \cite{PP3}, is an
improvement of Chen-Hacon and Jiang's theorems:
\begin{thm*}
 If $Z$ is a complex projective smooth variety of maximal
Albanese
dimension and general type, then the tetracanonical map
$\varphi_{|\omega_Z^{\otimes 4}|}$ birational.
\end{thm*}
The argument consists in showing that reducible divisors
$D_\alpha+D_{\alpha^\vee}$, with $\alpha\in \Pic^0(Z)$,\linebreak
$D_\alpha\in |\omega^{\otimes2}_Z\otimes\alpha|$ and
$D_{\alpha^\vee}\in |\omega_Z^{\otimes2}\otimes\alpha^\vee|$
separate points in a suitable open set of $Z$. The crucial point
is that, for all $\alpha\in \Pic^0(Z)$, the sections of
$\omega_Z^{\otimes2}\otimes\alpha$ passing trough a general point
of $Z$ have good generation properties. In order to achieve this
we use Pareschi-Popa theory of  $M$-regularity and continuous
global generation (\cite{PPI,PP3}), joint with a theorem of
Chen-Hacon \cite{Chen2001} on the fact that the variety
$V^0(\omega_Z)$ spans $\Pic^0(Z)$. \\
\indent It is worth to recall that if $Z$ is variety of maximal
Albanese
dimension,
then $\chi(\omega_Z)\ge 0$ (\cite{GL1}) and that
Chen-Hacon
proved in \cite{CH2007} that if $\chi(\omega_Z)\ge 1$ then, as
for curves, the
tricanonical map is always birational and this is sharp, as shown
by curves
of genus 2. Therefore our result is relevant only in the more
mysterious
case where $\chi(\omega_Z)=0$, which appears only in dimension
$\ge 3 $ (see
\cite{Ein1997}). In fact the author does not know any example of
varieties of
maximal Albanese dimension with non-birational tricanonical
map\footnote{I am deeply indebted to C. Hacon for
showing
me that the tricanonical maps of the varieties constructed by Ein
and Lazarsfeld in \cite[Example 1.13]{Ein1997} are always
birational. In a
similar way it is possible to see that the tricanonical map of
the variety constructed by Chen and Hacon in
\cite[Example at pg. 214]{CH2004} is birational.}
and it is
tempting to conjecture that the tricanonical map should be always
birational
for such varieties.\\

In the sequel, $Z$ will always be a smooth complex  variety of
general type and maximal Albanese dimension while $\omega_Z$
shall denote its dualizing sheaf. By $\mathrm{Alb}(Z)$ we will
mean the Albanese variety of $Z$. Given $\fas{L}$ an invertible
sheaf on a projective variety $Y$, then $\fas{J}(||\LL||)$ will
be the asymptotic
multiplier ideal sheaf associated to the complete linear series
$|\LL|$ (cfr. \cite{LazarsfeldIeII}). Finally, if $\F$ is a
coherent sheaf on $Y$, then by $h^i(\F)$ we will
mean $\dim
H^i(Y,\F)$. Given a linear system $V\subseteq|\LL|$, we will
say, with a slight abuse of notation,
that it is birational if the corresponding rational map
$\varphi_V$ is birational.

\subsubsection*{Acknowledgments} I am indebted to my advisor, G.
Pareschi, for many advices and valuable suggestions. I am also
very grateful to
C. Hacon, and M. Popa for carefully reading a first draft of
this paper and suggesting many emprovements and corrections.
Moreover, a particular thank must go to C. Hacon for explaining
to me how to work with some specific
examples. Finally it is a pleasure to thank A. Calabri, C.
Ciliberto, A. Rapagnetta, G. Pareschi and F. Viviani for very
interesting and
useful discussion.
\section{Setup and definitions}

In what follows we briefly recall some definitions
about $M$-regular sheaves and their generation property.
Given a coherent sheaf $\F$ on a smooth projective variety $Y$
we denote by $\mathrm{Bs}(\F)$ the \emph{non-generation locus
of $\F$}, i. e. the support of the cokernel of the evaluation
map\linebreak $H^0(Y,\F)\otimes\OO{Y}\longrightarrow\F $. If
$\F$ is a line bundle, then $\mathrm{Bs}(\F)$ is just the base
locus of $\F$.\\
The \emph{cohomological
support loci} of $\F$ are defined as:
$$V^i(Y,\F)%
:=\{\alpha\in\Pic^0(Y)\:|\:h^i(\F\otimes\alpha)>
0\}\subseteq\Pic^0(Y).$$
As for cohomology groups, we will occasionally suppress $Y$ from
the notation, simply writing $V^i(\F)$.

\begin{dfn}[{\cite[Definitions 2.1 and
2.10]{PPI},
}]
Let $\F$ a sheaf on a smooth projective variety $Y$.
\begin{enumerate}
 \item It is said to be $GV_1$ if
$\mathrm{Codim}\:
V^i(\F)>i$
for every $i>0$. In the case $Y$ is an abelian variety, $GV_1$
sheaves are called \emph{$M$-regular}.
\item Given $y\in Y$, we say that $\F$ is \emph{continuously
globally
generated at $y$} (in brief CGG at $y$) if the sum of the
evaluation maps
\[\mathcal{E}v_{U,y}:\bigoplus_{\alpha\in U}H^0(\F\otimes
\alpha)\otimes\alpha^{-1}\longrightarrow
\F\otimes\mathbb{C}({y})\]
is surjective for every $U$ non empty open set of $\Pic^0(Y)$.
\end{enumerate}
\end{dfn}
\begin{exa}\label{ese}
 In \cite[Proposition 2.13]{PPI} Pareschi and Popa proved that an
$M$-regular sheaf on an abelian variety is CGG at every point.
\end{exa}

\section{Proof of Theorem A}
We will prove this slightly more general statement:
\begin{thm}\label{thm: A}
 Let $Z$ be a smooth complex projective variety of maximal
Albanese dimension and of  general type, then, for every
$\alpha\in\Pic^0(Z)$, the linear system $|\omega_Z^{\otimes
4}\otimes\alpha|$ is birational.
\end{thm}
\noindent Note that also in the setting of
Jiang and Chen-Hacon what is proved is the birationality of
$\omega_Z^{\otimes5}\otimes
\alpha$ for all $\alpha\in\Pic^0(Z)$.\\
We will need the following easy lemma.
\begin{lem}\label{lem:effdiv}
 Let $E$ an effective divisor on an abelian variety $X$. Take
$T_1,\;\ldots,T_k$ subtori of $X$ such that they generate $X$ as
an abstract group and let $\gamma_i$, $i=1,\ldots,k$ some points
of $X$. Then $E\cap (\gamma_i+T_i)\neq\emptyset$ for at least one
$i$.
\end{lem}
\begin{proof}
 If $E$ is ample, then the statement is easily seen to be true.
Thus we can suppose, without a loss of generality, that $E$ is
not ample. Then there exists an
abelian variety $A\neq 0$, a surjective homomorphism
$f:X\longrightarrow A$ and an ample divisor $H$ on $A$ such that
$E=f^*H$. Suppose now that $E$ does not
intersect $T_i+\gamma_i$ for every $i$, it follows that $H$ does
not intersect $f(T_i+\gamma_i)$ for every $i$. Hence, since $H$
is ample, the image $T_i+\gamma_i$ through $f$ is a point, and
therefore $T_i$ is contained in the kernel of $f$  for every
$i$. But this is impossible since the
$T_i$'s generate $X$.
\end{proof}

\begin{proof}[Proof of Theorem \ref{thm: A}]\hfill\\
Let $a_Z:Z\longrightarrow \mathrm{Alb}(Z)$  Albanese map of $Z$.
Let us state the following claim.\\

\textit{Claim 1:} For the generic $z\in Z$
the sheaf
$a_{Z\:*}(\fas{I}_z\otimes
\omega_Z^{\otimes2}\otimes\fas{J}
(||\omega_Z||))$ is $M$-regular.\\

Before proceeding with the proof of the Claim 1 let us see how
it implies Theorem A, this argument follows the one Pareschi
and Popa in \cite{PP3}. First of all we simplify a bit the
notation by letting $\F:=\omega_Z^{\otimes
2}\otimes\fas{J} (||\omega_Z||)$.
Now observe that, since the Albanese morphism $a_Z$ is
generically finite, a well known extension of
Grauert-Riemenschneider vanishing theorem (see, for example
\cite[proof of Proposition 5.4]{PP3}) yields that, for any
$\alpha \in\Pic^0(Z)$, the higher direct images
$R^ia_{Z\:*}(\omega_Z^{\otimes
2}\otimes\alpha\otimes\fas{J}(||\omega_Z\otimes\alpha||))$
vanish. Therefore, for every $i\geq 0$ we get the equality
\[V^i(Z,\omega_Z^{\otimes
2}\otimes\alpha\otimes
\fas{J}(||\omega_Z\otimes\alpha||))=V^i(\mathrm{Alb}(Z),a_{Z\:*}
(\omega_Z^ {
\otimes 2}\otimes\alpha\otimes
\fas{J}(||\omega_Z\otimes\alpha||))).
\]
By Nadel vanishing for adymptotic multiplier ideals (see
\cite[Theorem 11.2.12]{LazarsfeldIeII}) the loci at the left
hand side are empty when $i>0$. In particular
$a_{Z\:*}(\omega_Z^{
\otimes 2}\otimes\alpha\otimes
\fas{J}(||\omega_Z\otimes\alpha||))$ is $M$-regular on
$\mathrm{Alb}(Z)$ and hence
CGG everywhere (Example \ref{ese}).  Since $a_Z$ is generically
finite, we get that
$\omega_Z^{
\otimes 2}\otimes\alpha\otimes
\fas{J}(||\omega_Z\otimes\alpha||)$ is CGG away from $W$, the
exceptional locus of $a_Z$. This means by definition that the
map
\begin{equation}\label{eq:popa}
\bigoplus_{\beta\in \Omega}H^0(\omega_Z^{
\otimes 2}\otimes\alpha\otimes
\fas{J}(||\omega_Z\otimes\alpha||)\otimes
\beta)\otimes\beta^{-1}\xrightarrow{ev_z}
\omega_Z^{
\otimes 2}\otimes\alpha\otimes
\fas{J}(||\omega_Z\otimes\alpha||)\otimes\mathbb{C}({z})
\end{equation}
is surjective for every $z\notin W$ and for every
$\Omega\subseteq\Pic^0(Z)$ non-empty open set. Now we recall that
the base locus of
$\fas{J}(||\omega_Z\otimes\alpha||)$ is contained in the base
locus of $\omega_Z^{\otimes 2}\otimes\alpha^{\otimes 2}$ for
every $\alpha\in\Pic^0(Z)$; hence we can consider the closed set
$W'=W\cup\{\text{zero locus of
}\fas{J}(||\omega_Z\otimes\alpha||)\}$. It follows easily from
\eqref{eq:popa} that, if $z\notin W'$ the map
\begin{equation*}
\bigoplus_{\beta\in \Omega}H^0(\omega_Z^{
\otimes 2}\otimes\alpha\otimes
\beta)\otimes\beta^{-1}\xrightarrow{ev_z}
\omega_Z^{
\otimes 2}\otimes\alpha\otimes\mathbb{C}({z})
\end{equation*}
is surjective for every $\Omega$ open dense set in $\Pic^0(Z)$
and thus
$\omega_Z^{\otimes2}\otimes\alpha$ is CGG away from $W'$.\\
\indent From now on, since we are supposing that the claim is
true, we can take $U$ a non-empty open
set of $Z$
such that for every $z\in U$ the sheaf
$a_{Z\:*}(\fas{I}_z\otimes \F)$ is $M$-regular and hence CGG.
Again, it follows that
$\fas{I}_z\otimes
\F$ is CGG away from $W$. Finally by
\cite[Proposition 2.12]{PPI}
$\fas{I}_z\otimes\F\otimes \omega_Z^{\otimes 2}\otimes\alpha$ is
globally generated away from $W'$ and therefore
$\fas{I}_z\otimes \omega_Z^{\otimes 4}\otimes \alpha$ is
globally generated away from $T=W'\cup\{\text{zero locus of
}\fas{J}(||\omega_Z||) \}$ and
we can conclude that $\omega_Z^{\otimes
4}\otimes\alpha$ is very ample outside a proper subvariety of $Z$
and hence
$\varphi_{|\omega_Z^{\otimes 4}\otimes \alpha|}$ is birational
for every $\alpha \in\Pic^0(Z)$. This prove the Theorem.\\

\indent Now we proceed with the proof Claim 1. Recall that, by
the Subtorus
Theorem (\cite[Theorem 0.1]{GL2} and \cite{S}), we
can write
\[V^0(\omega_Z)=\bigcup_{i=1}^k T_k +[\beta_k]\] 
where the $\beta_j$'s are  torsion points of $\Pic^0(Z)$ and the
$T_j$'s are subtori of $\Pic^0(Z)$. Hence, for all
$i$ and for all
$\alpha \in T_i$ the line bundle $\omega_Z\otimes
\alpha\otimes\beta_i$ has
sections, and therefore its base locus $Bs(\omega_Z\otimes
\alpha\otimes\beta_i)$ is a proper subvariety of Z. We take an
non-empty
Zariski open set $U$ contained in the complement of
\[W\cup
\bigcap_{i=1}^k\bigcup_{\alpha\in
T_i}\mathrm{Bs}(\omega_Z\otimes\alpha\otimes\beta_i). \]
Given $z\in U$ we want to prove that for every
$i\geq1$
\[\mathrm{Codim}\:V^i(a_{Z\:*}(\fas{I}_z\otimes
\F))\geq i+1.\]
Consider the following short exact sequence:
\begin{equation}\label{corta}\escorta{\fas{I}_z\otimes
\F\otimes\gamma}{\F\otimes\gamma}{\F\otimes\gamma\otimes\C(z)},
\end{equation}
where $\gamma\in\Pic^0(Z)$.
Since $z$ does not belong to the exceptional locus of the
Albanese map of $Z$, by pushing forward for $a_Z$ we still get a
short exact
sequence:\begin{equation}\label{eq:short}\escorta{a_{Z\:*}(\fas{
I}
_z\otimes
\F\otimes\gamma)}{a_{Z\:*}(\F\otimes\gamma)}{
a_{Z\:*}(\F\otimes\gamma\otimes\C(z)) } .\end{equation}
In particular $R^1a_{Z\:*}(\fas{
I}
_z\otimes
\F\otimes\gamma)=0$. Since, as already observed, an extension of
Grauert-Riemenschneider vanishing, yields the vanishing the
 $R^ja_{Z\:*}(\omega_Y^{\otimes
2}\otimes\fas{J}(||\omega_Y||)\otimes\gamma)$ for every $j>0$ and
every $\gamma\in\Pic^0(Z)$, we have that $R^ja_{Z\:*}(\fas{
I}
_z\otimes
\F\otimes\gamma)=0$ for every $j>0$ and every $\gamma$.
Therefore, for every $j\geq 0$ 
\begin{equation}\label{equality}
V^i(\mathrm{Alb(Z)},a_{Z\:*}(\fas{I}_z\otimes
\F))=V^i(Z,\fas{I}_z\otimes
\F).
\end{equation}
Now we look at the long cohomology sequence of the short exact
sequence \eqref{corta}. By Nadel vanishing we have that, for
every $i\geq 1$ and $\gamma\in\Pic^0(Z)$,
\[ H^i(Z,
\omega_Y^{\otimes
2}\otimes\fas{J}(||\omega_Y||)\otimes\alpha)%
{=}0 
\]
and therefore, by \eqref{equality}
the loci $V^i(a_{Z\:*}(\fas{I}_z\otimes
\F))$ are empty for every $i\geq
2$.\\
 It remains to see that
$\mathrm{Codim}V^1(a_{Z\:*}(\fas{I}_z\otimes\F))>1$. Before
proceeding further we state the following:\\

\indent\textit{Claim 2:} For every $j=1,\ldots,k$,
$V^1(\fas{I}_z\otimes
\F)\cap (T_j+[\beta_j^{\otimes 2}])=\emptyset.$\\

The reader may observe that, since by a result of Chen and Hacon
(\cite[Theorem 1]{Chen2001}) the $T_j$'s generate $\Pic^0(Z)$,
this, together with Lemma
\ref{lem:effdiv}, implies directly Claim 1. In fact  the locus
$V^1(\fas{I}_z\otimes\F)$ is a proper subvariety of
$\Pic^0(Z)$ that cannot possibly contain a
divisor, hence its codimension
shall be greater than 1. By \eqref{equality} also
$\mathrm{Codim}V^1(a_{Z\:*}(\fas{I}_z\otimes\F))>1$.\\

\indent To prove Claim 2 we reason in the following way.  First
we remark that there is an
alternative description of the locus
$V^1(\fas{I}_z\otimes\F)$: in fact by \eqref{corta} and the
vanishing of $H^1(Z,\F\otimes\gamma)$ for every $\gamma$ (again
due to Nadel vanishing) we have that
$H^1(Z,\fas{I}_z\otimes\F\otimes\gamma)\neq 0$ if and only if
the map
$H^0(Z,
\F\otimes\gamma)\longrightarrow\F\otimes\gamma\otimes\C(z)$ is
not surjective. Therefore we can write
\[V^1(\fas{I}_z\otimes\F)=\{\gamma\in\Pic^0(Z)\text{
such that } z\notin \mathrm{Bs}(\F\otimes\gamma)\}.\]
Since we chose $z\notin\bigcap_{\alpha\in
T_j}\mathrm{Bs}(\omega_Z\otimes\alpha\otimes\beta_i)$ there
exists $V_j$ a dense open set of $T_j$  such that, for every 
$\delta\in V_j$, $z$ is not in
$\mathrm{Bs}(\omega_Z\otimes\beta_i\otimes\delta)$. Take
$\gamma\in T_j$; the intersection $V_j\cap(\gamma-V_j)$
is still a non-empty open set of $T_j$. For $\delta\in
V_j\cap(\gamma-V_j)$ we write
\[\omega_Z^{\otimes 2}\otimes\beta_j^{\otimes 2}\otimes \gamma
\simeq(\omega_Z\otimes\beta_j\otimes\delta)\otimes
(\omega_Z\otimes\beta_j\otimes\gamma\otimes\delta^{-1});\]
since, thanks to our choice of $\delta$, both
$\omega_Z\otimes\beta_j\otimes\delta$ and
$\omega_Z\otimes\beta_j\otimes\gamma\otimes\delta^{-1}$ are
generated
at $z$, also the left hand side is. Therefore
$z\notin\mathrm{Bs}(\omega_Z^{\otimes
2}\otimes\beta_j^{\otimes
2}\otimes\gamma)$.
Now by
\cite[Lemma
6.3(b)]{PP3}, the
zero locus of
$\fas{J}(||\omega_Z||)$
is contained in the base
locus of
$\omega_Z^{\otimes
2}\otimes\alpha$ for
every $\alpha
\in\Pic^0(Z)$; thus we
have the equality $\mathrm{Bs}(\omega_Z^{\otimes
2}\otimes\beta_j^{\otimes
2}\otimes\gamma)=\mathrm{Bs}(\F\otimes\beta_j^{
\otimes
2}\otimes \gamma)$ and, consequently,
$\gamma\otimes\beta_j^{\otimes 2}\notin
V^1(\fas{I}_z\otimes\F)$ and the claim is proved.
\end{proof}



\end{document}